\documentclass{amsart}
\usepackage{amsmath,amssymb}

\newtheorem{theorem}{Theorem}
\newtheorem{conjecture}[theorem]{Conjecture}
\newtheorem{problem}[theorem]{Problem}

\newcommand{\C}{\mathbb{C}}

\newcommand{\hatM}{\widehat{M}}

\begin{document}

\title{Holomorphic Hulls for Compact 3-Manifolds}
\author{Ali M. Elgindi}
\address{Institute of Mathematics, Henan Academy of Sciences}
\email{ali.m.elgindi@gmail.com}

\maketitle

\begin{abstract}
We demonstrate the different possible structures for holomorphic hulls for embeddings of compact real 3-manifolds $M \hookrightarrow \mathbb{C}^3$ along the set of complex tangents $\gamma$. Using our previous work [1], we can construct embeddings with any prescribed link and 2-plane field along it, as well as a prescription of angles at each point that determines the Bishop invariant. We show that elliptic points ($\gamma < \frac{1}{2}$) produce analytic discs filling a Levi-flat hypersurface, hyperbolic points ($\gamma > \frac{1}{2}$) add no local hull structure, and parabolic points ($\gamma = \frac{1}{2}$) may develop  a  "delicate" Jöricke  `onion' structure. We illustrate these phenomena with explicit examples of parabolic knots and links with mixed components.
\end{abstract}

\section{Introduction}

Complex tangents to an embedding $M^k \hookrightarrow \C^n$ are points $x \in M$ where the tangent space contains a complex line. For embeddings of 3-manifolds into $\C^3$, complex tangents generically form a link $L \subset M$ [1,2].

The Bishop invariant $\gamma(x) \in [0,\infty]$ assigns to each non-degenerate complex tangent a number that determines the local behavior of the holomorphic hull $\hatM$. When $\gamma < \frac{1}{2}$ (elliptic), there exists a one-parameter family of analytic discs with boundaries in $M$ accumulating at $x$, filling the hull locally [3]. When $\gamma > \frac{1}{2}$ (hyperbolic), no such discs exist. The parabolic case $\gamma = \frac{1}{2}$ is more delicate, with Jörickee's  theory revealing possible ``onion'' structures [6,7].

In our previous work [1], we proved:

\begin{theorem}[Elgindi, 2025]
Let $M$ be a closed orientable 3-manifold, $L \subset M$ a null-homologous link, and $\eta$ a smooth oriented 2-plane field along $L$. Then there exists a smooth embedding $M \hookrightarrow \C^3$ whose complex tangents are exactly $L$ and whose holomorphic tangent spaces along $L$ are $\eta$. Moreover, the Bishop invariant $\gamma(x)$ is determined by the angle $\theta(x)$ between $T_xL$ and the normal to $\eta_x$:
\begin{itemize}
\item If $\theta \in [0, \frac{\pi}{2})$ then $\gamma > \frac{1}{2}$ (hyperbolic);
\item If $\theta \in (\frac{\pi}{2}, \pi]$ then $\gamma < \frac{1}{2}$ (elliptic);
\item If $\theta = \frac{\pi}{2}$ then $\gamma = \frac{1}{2}$ (parabolic).
\end{itemize}
\end{theorem}

The goal of this paper is to describe, for such embeddings, the structure of the holomorphic hull $\hatM$ in terms of the data $(L, \theta)$.

\section{Background}

\subsection{Bishop's Theorem and Analytic Discs}

Bishop [3] proved that near an elliptic complex tangent ($\gamma < \frac{1}{2}$), there exists a one-parameter family of analytic discs with boundaries in $M$. These discs converge to the complex tangent point. Kenig and Webster [8] extended this result, showing that the discs fill a smooth Levi-flat hypersurface.

For hyperbolic points ($\gamma > \frac{1}{2}$), no such discs attach locally. The hull coincides with $M$ in a neighborhood of such points.

\subsection{Jöricke's Onion Structures}

At parabolic points ($\gamma = \frac{1}{2}$),  Jöricke [6,7] discovered that the hull may develop a nested structure of Levi-flat shells, which she termed as an ``onion.'' The number of shells is determined by a certain index, which can be computed from the higher-order terms of the defining function.

\par\ \par

At a parabolic complex tangent point ($\gamma = \frac{1}{2}$), Jöricke defined an integer index $k \ge 0$ as follows. In suitable local holomorphic coordinates $(z_1, z_2, z_3)$, the embedding can be written as:
\[
z_2 = \bar{z}_1 + \frac{1}{2}(z_1^2 + \bar{z}_1^2) + \Phi(z_1, \bar{z}_1) + \dots
\]
where $\Phi$ contains higher-order terms. The index $k$ is the order of vanishing of a certain function derived from $\Phi$. Equivalently, $k$ is the number of nested Levi-flat shells (``onion layers'') in the holomorphic hull near the point. When $k = 0$, the hull coincides with $M$ locally; when $k > 0$, the hull contains $k$ distinct shells.

For a detailed computation of the index, see \cite{Joricke1990, Joricke1994}.

\subsection{The Construction}

Our holomorphic embedding theorem in [4] provides complete flexibility: for any link $L$ and any 2-plane field $\eta$ along $L$, one can realize $L$ as the complex tangent locus with holomorphic tangent spaces $\eta$. The angle $\theta$ between $T_xL$ and the normal to $\eta_x$ determines $\gamma(x)$ via the classification above.

\section{Main Results}
\subsection{Elliptic Components}
\begin{theorem}[Elliptic Components] Let $M \hookrightarrow \mathbb{C}^3$ be an embedding as in [1] with a link component $K \subset L$ where $\theta > \frac{\pi}{2}$ gives elliptic type with Bishop invariant $\gamma < \frac{1}{2}$. Then the holomorphic hull $\hatM$ contains a smooth Levi-flat hypersurface $\mathcal{H}$ with its boundary lying on the knot of complex tangents: $\partial\mathcal{H} \subset K$ and $K$ is itself foliated by analytic discs whose boundaries lie in $M$.
\end{theorem}
\begin{proof}
By Bishop's theorem in [3], for each elliptic complex tangent point: $x \in K, \gamma (x)< \frac{1}{2}$) there exists a one-parameter family of analytic discs ${D_{x,t},\{t \in (0,\epsilon]}\}$ with boundaries in $M$ converging to $x$. Kenig and Webster [8] proved that the union of these discs near $x$ is a smooth Levi-flat hypersurface. Since the embedding in [4] varies smoothly along $x \in K$, these local families glue together to form a global Levi-flat hypersurface $\mathcal{H}$ with its boundary $\partial \mathcal{H} \subset K$. The discs foliate $\mathcal{H}$ by construction.
\end{proof}
\subsection{Hyperbolic Components}
\begin{theorem}[Hyperbolic Components]
Let $M \hookrightarrow \mathbb{C}^3$ be as above with a complex tangent link component $K$ where the analytic angle $\theta < \frac{\pi}{2}$, making that the points are of  Bishop hyperbolic type, $\gamma > \frac{1}{2}$. (by our work in [4])  Then there exists a neighborhood U of $K$ such that: $\hat{M} \cap U = M \cap U$; i.e., no other local hull formation occurs.\end{theorem}\begin{proof}Bishop [5] proved that for hyperbolic points ($\gamma > \frac{1}{2}$) no analytic discs attach to $M$ in a neighborhood of the point. The hull $\hat{M}$ is defined as the set of points that cannot be separated from $M$ by holomorphic functions. Without attaching discs, the hull cannot extend beyond $M$ locally. Hence, there exists a neighborhood $U$ of $K$ such that $\widehat{M \cup U} = M \cup U$

\end{proof}
\subsection{Parabolic Components and Jöricke's Onion}
\begin{theorem}
[Parabolic Components]Let $M \hookrightarrow \mathbb{C}^3$ be as above with a link component $K$ where $\theta = \frac{\pi}{2}$ , and as such has Bishop parabolic type: $\gamma = \frac{1}{2}$. The
\begin{enumerate}\item If the Jöricke index vanishes, the hull is locally trivial: $\hat{M} = M$ near $K$
\item If the Jöricke index is positive, the hull exhibits nested Levi-flat shells (an ``onion'') centered along $K$, as described in [6,7].
\end{enumerate}
\end{theorem}
\begin{proof}
Jöricke [6,7] proved that at a parabolic point ($\gamma = \frac{1}{2}$) the local hull structure is determined by a discrete index. If the index is zero, the hull coincides with $M$ in a neighborhood of the point. If the index is $k > 0$, the hull contains $k$ nested Levi-flat nested shells (an onion) centered at the point. By the smoothness of the construction in [1], these local structures vary smoothly along $K$, yielding a global onion structure along the entire component
\end{proof}
\subsection{Transitions Along a Single Knot}
\begin{theorem}[Transitions]Let $K$ be a single knot component of the complex tangent link $L\subset M$ such that the holomorphic angle $\theta$ varies along $K$ (see [4]), taking values both $> \frac{\pi}{2}$ and $< \frac{\pi}{2}$. By our work in[4], this corresponds respectively with the hyperbolic and elliptic Bishop invariant construction ($\gamma>\frac{1}{2}, <\frac{1}{2}$). Furthermore, at points where the holomorphic angle is $\theta = \frac{\pi}{2}$ (corresponding to parabolic Bishop $\gamma=\frac{1}{2}$), the hull $\hat{M}$ develops a corner or singularity\end{theorem}
\begin{proof}On the elliptic side of a parabolic point $p \in K$, Theorem 3.1 gives a Levi-flat tube attached to $K$. On the hyperbolic side, Theorem 3.2 gives no attachment. The boundary of the tube must therefore lie exactly at $p$. This abrupt change in the hull's geometry produces a corner or singularity at $p$. The exact nature of the singularity depends on the Jöricke index at $p$; for index $k = 0$ the corner is sharp, while for $k > 0$ the onion shells smooth the transition.\end{proof}

\section{Examples}

\subsection{Purely Parabolic Knot}

Consider $f(z,w) = \frac{1}{2}(\bar{z}^2 + \bar{w})$ on $S^3$. By [2], the complex tangents form a great circle $K = \{z=0, |w|=1\}$  with Bishop invariant $\gamma = \frac{1}{2}$ everywhere. By the parabolic theorem above, the hull $\hat{M}$   is either trivial or exhibits an onion structure. A direct computation shows the Jöricke index vanishes, so $\hatM = M$ locally.

\subsection{Linked Circles with Mixed Types}

Consider the embedding from Example 5 of [2] with two linked circles:
\begin{itemize}
\item $K_1$: elliptic ($\gamma \approx 0.09$, $\theta > \frac{\pi}{2}$)
\item $K_2$: hyperbolic ($\gamma > \frac{1}{2}$, $\theta < \frac{\pi}{2}$)
\item Linking number $= 2$
\end{itemize}
By the elliptic theorem, a Levi-flat sheet $\mathcal{H}$ attaches to $K_1$. By the hyperbolic theorem, no sheet attaches to $K_2$. Because the components are linked, $\mathcal{H}$ wraps around $K_2$ twice. The boundary of $\mathcal{H}$ thus contains both $K_1$ and, after winding twice, $K_2$.

\subsection{The $\alpha$-Family with Varying $\alpha$}

Take $f(z,w) = \alpha(w)\bar{z}^2 + \bar{w}$ where $\alpha(w)$ varies with $w$ such that $|\alpha(w)|$ crosses $\frac{1}{2}$ at two points. Then $K = \{z=0\}$ is a single knot with:
\begin{itemize}
\item Elliptic arcs where $|\alpha(w)| < \frac{1}{2}$;
\item Hyperbolic arcs where $|\alpha(w)| > \frac{1}{2}$;
\item Parabolic points where $|\alpha(w)| = \frac{1}{2}$.
\end{itemize}
The hull develops corners at the parabolic points, separating regions where discs attach (elliptic) from those where they do not (hyperbolic).

\section{A Hull as a hat capped by two onions} 
We now construct an explicit embedding of $S^3$ into $\mathbb{C}^3$ with complex tangents forming a single knot $K$ that consists of an elliptic arc, a parabolic point, and a hyperbolic arc. This example illustrates the transitioning between between the three types and shows how the holomorphic hull $\widehat{M}$ changes from a Levi-flat tube to an onion cap on top of $M$ finally to nothing. Let $K \subset S^3$ be the great circle given by $\{z = 0, |w| = 1 \} \subset \mathbb{C}^3$ parametrized by $w = e^{i\phi}$ with $\phi \in [0, 2\pi)$.  We now define:

$ \alpha(\phi) = \frac{1}{2} + \cos\phi$ 

Then: 
\begin{itemize}
\item For $\phi \in [0, \frac{\pi}{2})$, we have that: $\alpha(\phi) > \frac{1}{2}$ (elliptic type). \item At $\phi = \frac{\pi}{2}, \alpha(\phi) = \frac{1}{2}$ (parabolic point)
\item For \\$\phi \in (\frac{\pi}{2}, \frac{3\pi}{2})$. we have $\alpha(\phi) < \frac{1}{2}$ (hyperbolic type)
. \item At $\phi = \frac{3\pi}{2}$, we again have that: $\alpha(\phi) = \frac{1}{2}$  forming another parabolic point.
\item For  $\phi \in  (\frac{3\pi}{2}, 2\pi)$, we have that $\alpha(\phi) > \frac{1}{2}$ (elliptic type). 

\end{itemize} 
\par\ \par\
Consider then the embedding $F: S^3 \hookrightarrow \mathbb{C}^3$ given by:
$$ F(z,w) = (z, w, \alpha(\phi) \bar{z}^2 + \bar{w})$$
where $\phi$ is the angular coordinate of $w$. By our construction in [2], the complex tangents of  $F$ form exactly the great circle $K$ 
\par\ \par\
\subsubsection{The Holomorphic hull} 
\begin{itemize} \item On the elliptic arcs $\phi \in [0, \frac{\pi}{2}) \cup  (\frac{3\pi}{2}, 2\pi)$), by our previous theorem we have  two connected Levi-flat tubes attached to K (by onions). 
\item At the parabolic points $\phi = \frac{\pi}{2}, \frac{3\pi}{2}$, Theorem 3.3 gives a single Levi-flat disc (onion) capping the adjacent tube at each point. Hence, we get two onion caps for the analytic tube  overseeing the elliptic arc we had above.
\item On the hyperbolic arc $\phi \in  (\frac{\pi}{2}, \frac{3\pi}{2})$, we know that there is no extra addition the hull. 
\item Therefore, the structure of the hull is given as an analytic tube with two onions at the boundaries, in essence giving the structure of an analytic "hat" with two "caps of onions" on the boundaries of the tube that is overseeing the elliptic arc of complex tangents on the manifold $M$.
\item Thus the structure of the hull is an analytic tube \\$T$ attached along the elliptic arc, with two onions $D_1$ and $D_2$ capping its two ends at the parabolic points of the elliptic arc. We can write this then as: $\widehat{M} = M \cup T \cup D_1 \cup D_2 \cup \emptyset$ where T is the analytic tube attached along the elliptic arc, and $D_1, D_2$ are the Jörickee discs (onions) at the parabolic points and $\emptyset$ is the contribution from the hyperbolic points (nothing).
\end{itemize} 
\begin{quote}
    
This example shows the transition from elliptic (tube) to parabolic (cap) to hyperbolic (nothing) on a single knot. The Jöricke index $k = 1$ gives a single shell; higher indices would give nested shells. 
\end{quote} 

\section{Conjectures and Open Problems}

\begin{conjecture}[Linking and Hull Containment]
Let $K_1$ be elliptic and $K_2$ hyperbolic in a link $L = K_1 \cup K_2$ with linking number $\ell k(K_1, K_2) \neq 0$. Then $K_2 \subset \hatM$. If $\ell k = 0$, then $K_2 \not\subset \hatM$.
\end{conjecture}

\begin{conjecture}[Isolated Complex Tangents]
An isolated complex tangent to an embedding $M^3 \hookrightarrow \C^3$ is necessarily degenerate. Moreover, a small perturbation resolves it into an elliptic or hyperbolic circle.
\end{conjecture}

\begin{problem}
Determine the precise relationship between the Jöricke index at a parabolic point and the topology of the resulting onion structure. Does the index correspond to the number of nested shells?
\end{problem}

\begin{problem}
Extend the results of this paper to higher dimensions: embeddings $M^{2n-1} \hookrightarrow \C^{2n-1}$ for $n \geq 3$.
\end{problem}
\begin{quote}
    
\end{quote}
\begin{quote}
    Note that we already extended our results to higher dimensional spheres and Heisenberg groups: $\mathbb{H}^{2n-1} \hookrightarrow \mathbb{C}^{2n-1}$ in our paper [3].
\end{quote}

\end{document}